\documentclass{article}

\usepackage[notref,notcite,color]{showkeys}

\usepackage[centertags]{amsmath}
\usepackage{amsfonts}
\usepackage{amssymb}
\usepackage{amsthm}
\usepackage{newlfont}
\usepackage{mathtools}
\usepackage[top=1in, bottom=1in, left=1in, right=1in]{geometry}
\usepackage{tikz}

\theoremstyle{plain}
\newtheorem{theorem}{Theorem}[section]
\newtheorem{proposition}[theorem]{Proposition}
\newtheorem{corollary}[theorem]{Corollary}
\newtheorem{lemma}[theorem]{Lemma}

\theoremstyle{definition}
\newtheorem{definition}[theorem]{Definition}

\newtheorem{example}{Example}[section]

\newcommand{\R}{\mathbb{R}}

\newcommand{\Z}{\mathbb{Z}}
\newcommand{\C}{\mathbb{C}}
\newcommand{\Q}{\mathbb{Q}}

\newcommand{\Hidden}[1]{}

\newcommand{\1}{\mathbf 1}

\DeclareMathOperator{\Tr}{Tr}

\begin{document}

\title{Pretty good state transfer in graphs with an involution}
\author{Mark Kempton\footnote{Center of Mathematical Sciences and Applications, Harvard University, Cambridge MA, mkempton@cmsa.fas.harvard.edu}~~~~~Gabor Lippner\footnote{Department of Mathematics, Northeastern University, Boston MA, g.lippner@neu.edu}~~~~~Shing-Tung Yau\footnote{Department of Mathematics, Harvard University, Cambridge MA, yau@math.harvard.edu}}
\date{}

\maketitle

\begin{abstract}
We study pretty good quantum state transfer (i.e., state transfer that becomes arbitrarily close to perfect) between vertices of graphs with an involution in the presence of an energy potential.  In particular, we show that if a graph has an involution that fixes at least one vertex or at least one edge, then there exists a choice of potential on the vertex set of the graph for which we get pretty good state transfer between symmetric vertices of the graph.  We show further that in many cases, the potential can be chosen so that it is only non-zero at the vertices between which we want pretty good state transfer.  As a special case of this, we show that such a potential can be chosen on the endpoints of a path to induce pretty good state transfer in paths of any length.  This is in contrast to the result of \cite{us}, in which the authors show that there cannot be perfect state transfer in paths of length 4 or more, no matter what potential is chosen.
\end{abstract}

\section{Introduction}

Given a graph $G$, the \emph{discrete Schr\"odinger equation} on $G$ is given by
\begin{equation}\label{eq:shrod}
\frac{d}{dt}\varphi_t = iH\varphi_t
\end{equation}
where $\varphi_t:V(G)\rightarrow \C$ is a function on the vertex set of $G$, and $H$ is the graph Hamiltonian.  Then equation (\ref{eq:shrod}) describes the evolution of the quantum state of a particle on the graph $G$ with time.  In this paper, we take $H = A+Q$ where $A$ is the adjacency of $G$, and $Q$ is an arbitrary diagonal matrix.  We think of $Q$ as a function $Q:V(G)\rightarrow\R$ that represents energy at each vertex. The function $Q$ is called a \emph{potential} on the vertex set.

\begin{definition}
We say that there is \emph{perfect state transfer} from vertex $u$ to vertex $v$ if, given the initial condition $\varphi_0 = \1_u$, there is some time $T$ at which the solution to (\ref{eq:shrod}) satisfies $|\varphi_T| = \1_v$, that is, $|\varphi_T(v)| = 1$ and $\varphi_T(x) = 0$ for $x\neq v$.  
\end{definition}

In \cite{us}, the authors studied perfect state transfer on graphs in this setting.  In particular, it has been proved that for paths of length at least four, there can be no perfect state transfer between endpoints of the path, no matter what potential one puts on the path.  In the current paper, we study a relaxation of perfect state transfer called \emph{pretty good state transfer}, which  does not require the transfer to be perfect ever, only to get arbitrarily close.  

\begin{definition}
We state that there is \emph{pretty good state transfer} from vertex $u$ to $v$ if, given the initial condition $\varphi_0 =\1_u$, for every $\epsilon>0$ there is some time $T$ such that $|\varphi_T(v)| > 1-\epsilon$.
\end{definition}

It was first observed in \cite{Casaccino2009} via numerical evidence that there might be pretty good state transfer in paths with the appropriate choice of potential at the endpoints. Our first main result confirms this observation.  


\begin{theorem}\label{thm:PGSTpath}
Given a path $P_N$ of any length, there is some choice of $Q$ such that by placing the value $Q$ as a potential on each endpoint of $P_N$ there is pretty good state transfer between the endpoints.
\end{theorem}

In \cite{Godsil2012} it is shown that in the absence of potential, pretty good state transfer occurs between endpoints of paths if and only if the number of nodes is $p-1$, $2p-1$ where $p$ is a prime, or if the number of nodes is $2^m-1$.  In \cite{Coutinho2016}, pretty good state transfer between internal nodes of paths is investigated, and surprisingly, there are cases where pretty good state transfer occurs between internal nodes of paths, but not the endpoints. \cite{vanBommel2016} proceeds to give a complete characterization of all cases when pretty good state transfer can occur between any nodes in paths (again, without potential), showing that those cases found in  \cite{Godsil2012} and \cite{Coutinho2016} are the only possibilities. In \cite{PGST}, pretty good state transfer in paths is studied using the Laplacian rather than the adjacency matrix, again without potential.  Our Theorem \ref{thm:PGSTpath} is thus surprising because, in the presence of potential at each endpoint, there is no condition on the number of vertices in the path.

We also prove the following very general result on graphs with an involution (an order two automorphism of the graph).

\begin{theorem}\label{thm:PGSTinvol}
Let $G$ be a connected graph with an involution $\sigma$, let $H$ denote the graph Hamiltonian, and let $Q: V(G)\rightarrow \R$ be a potential on the vertex set satisfying $Q(x) = Q(\sigma x)$ (so that $\sigma$ is also an involution of $H$).  Let $u$ and $v$ be vertices with $v = \sigma u$ and $u\neq v$.  Then if $\sigma$ fixes any vertices or any edges of $G$, then there is a choice of potential $Q$ for which there is pretty good state transfer from $u$ to $v$.
\end{theorem}

 Critical in this result is the observation that for a graph with an involution, if the potential that we put on the vertices is symmetric across the involution, then the symmetry of the graph naturally gives rise to a factorization of the characteristic polynomial of the graph Hamiltonian into two factors.  We will, throughout, refer to these factors as $P_+$ and $P_-$ (for reasons that will become clear later).  Understanding these factors will be a major key in our proofs.  We will see that $Q$ can always be chosen so that $P_+$ is an irreducible polynomial, which, as we will see, is an important element in the proofs.

The primary ingredient in our proofs is Lemma \ref{lem:eig}, which characterizes pretty good state transfer on graphs in terms of a condition on the eigenvectors of the Hamiltonian, as well as a number theoretic condition on its eigenvalues.   As we will see, graph with an involution will always automatically satisfy the condition on the eigenvectors, so the work in our proofs involves analyzing the number theoretic properties of the eigenvalues of the Hamiltonian.  

We will also discuss cases when a graph with an involution has no fixed points or edges.  These cases are harder to deal with, and many interesting scenarios can arise.

\section{Preliminaries}

Given a graph $G$ with $n$ vertices let $H = A +Q$ denote the graph Hamiltonian, where $A$ is the adjacency matrix, and $Q=diag(Q_1,\cdots,Q_n)$ a diagonal matrix with real entries.  Let $\varphi_0:V(G) \rightarrow \C$ be a complex-valued function on the vertex set of $G$ satisfying $||\varphi_0||_2 = 1$.  Define 
\[ \varphi_t(x) = e^{itH}\varphi_0(x)\]
and observe that $\varphi_t$ is a solution of (\ref{eq:shrod}).  We will denote $U(t) = e^{itH}$.  Note that the exponential of the matrix is given by 
\[
U(t) = e^{itH} = \sum_\lambda e^{it\lambda}xx^T
\]
where the sum is taken over eigenvalues $\lambda$ of $H$ and $x$ is the corresponding unit eigenvector.  In particular, note that 
\begin{equation}\label{eq:u0}
I = U(0) = \sum_\lambda xx^T.
\end{equation}
In addition, since $H$ is symmetric, each $\lambda$ above is real, and each $x$ can be assumed to have all real entries.  Observe also that $U(t)$ is a unitary matrix for all $t$, and therefore $||\varphi_t||^2 =  1$ for all $t$. Let $\1_u$ denote the indicator vector for vertex $u$. We say that there is \emph{pretty good state transfer} from $u$ to $v$, if, for any $\epsilon>0$, there is a time $T$ such that 
\[
||U(T)\1_u -  \gamma\1_v || < \epsilon
\]
for some $\gamma\in\C$ with $|\gamma|=1$. Equivalently, for all $\epsilon>0$, there is a time $T$ with
\[|U(T)_{u,v}| > 1-\epsilon.\]

\begin{lemma}[Kronecker]\label{lem:Kron}
Let $\theta_0,...,\theta_d$ and $\zeta_0,...,\zeta_d$ be arbitrary real numbers.  For an arbitrarily small $\epsilon$, the system of inequalities
\[
|\theta_ry - \zeta_r| < \epsilon ~~ (mod ~2\pi), ~~ (r=0,...,d),
\]
has a solution $y$ if and only if, for integers $\ell_0,...,\ell_d$, if
\[
\ell_0\theta_0+\cdots+\ell_d\theta_d = 0,
\]
then
\[
\ell_0\zeta_0 + \cdots + \ell_d\zeta_d \equiv 0 ~~ (mod ~2\pi).
\]
\end{lemma}

The following lemma is derived from results in \cite{PGST}.  We give a proof for completeness.

\begin{lemma}\label{lem:eig}
Let $u,v$ be vertices of $G$, and $H$ the Hamiltonian.  Then pretty good state transfer from $u$ to $v$ occurs at some time if and only if the following two conditions are satisfied:
\begin{enumerate}
\item Every eigenvector $x$ of $H$ satisfies either $x(u) = x(v)$ or $x(u) = -x(v)$.
\item Let $\{\lambda_i\}$ be the eigenvalues of $H$ corresponding to eigenvectors with $x(u) = x(v)\neq0$, and $\{\mu_j\}$ the eigenvalues for eigenvectors with $x(u) = -x(v)\neq0$.  Then if there exist integers $\ell_i$, $m_j$ such that if
\begin{align*}
\sum_i \ell_i\lambda_i +\sum_j m_j\mu_j = 0\\
\sum_i\ell_i +\sum_j m_j =0
\end{align*}
then
\[
\sum_i m_i \text{ is even.}
\]
\end{enumerate}
\end{lemma}
\begin{proof}
Suppose that pretty good state transfer occurs from vertex $u$ to vertex $v$.
Let $\{x_i\}$ be a set of orthonormal eigenvectors of $H$, and $\{\lambda_i\}$ the corresponding eigenvalues.  Then we can write
\[
U(t) = \sum_{i=1}^n e^{it\lambda_i}x_ix_i^T.
\]
Then for some $\gamma$ of modulus 1, for every $\epsilon>0$ there is some $t$ for which
\[
\left\|\sum_j e^{it\lambda_j}x_jx_j^T\mathbf 1_u - \gamma\mathbf 1_v \right\| < \epsilon.
\]
Multiplying by the projection matrix $x_kx_k^T$ for some $k$, it is not hard to see that this implies that 
\[
\left| e^{it\lambda_k}x_k(u) - \gamma x_k(v) \right| < \epsilon.
\]
Since $|e^{it\lambda_k}| = |\gamma| = 1$, the triangle inequality implies that 
\[
\big||x_k(u)| - |x_k(v)|\big| < \epsilon.
\]
This is true for all $\epsilon>0$, so we see that $|x_k(u)| = |x_k(v)|$, and $x_k$ can be taken to be a real vector, since $H$ is real and symmetric. Thus, $x_k(u) = \pm x_k(v)$, giving the first condition.

To see the second condition, for pretty good transfer to occur, for every $\epsilon>0$ we must have a time $t$ for which
\[
\left|\sum_j e^{it\lambda_j}x_j(u)^2 - \sum_je^{it\mu_j}x_j(u)^2 \right| > 1-\epsilon.
\]
From (\ref{eq:u0}), we know that $\sum x_j(u)^2=1$, so for this sum to be close to 1, we need the phases of the $e^{it\lambda_j}$ and $e^{it\mu_j}$ coefficients to be close to lining up to point in the same direction.  In other words, an equivalent formulation is for any $\epsilon>0$, there is a time $t$ such that 
\[
e^{it\lambda_i} = e^{i(t\lambda_0+\epsilon)} \text{ and } e^{it\mu_j} = -e^{i(t\lambda_0+\epsilon)}
\]
for all $i$ and $j$.  That is, for every $\epsilon>0$ there is a time $t$ such that
\[\begin{split}
\left|(\lambda_i - \lambda_0)t - 0\right| &< \epsilon~~~~(mod~2\pi)\\
\left|(\mu_j - \lambda_0)t - \pi\right| &< \epsilon~~~~(mod~2\pi).
\end{split}\]
Taking $\zeta_i$ to be 0 for the first set of equations, and $\pi$ for the second set, Lemma \ref{lem:Kron} implies that this has a solution $t$ if and only if, given integers $\ell_i$ and $m_j$ with 
\[
\sum_i \ell_i(\lambda_i - \lambda_0) + \sum_j m_j(\mu_j - \lambda_0) = 0,
\]
then we must have \[\sum_j m_j\pi \equiv 0~~~~(mod~2\pi).\]

Rearranging the first equation and dividing by $\pi$ in the second gives the condition of the lemma.

\end{proof}

We remark that the second condition in the preceding lemma implies that we do not get pretty good state transfer if there is a repeated eigenvalue belonging to both different types.  Indeed, if $\lambda_r=\mu_s$ for some $r,s$, then we can simply choose all the $\ell_r = 1$ and $m_s = -1$ and all other $\lambda_i$'s and $\mu_j$'s to be 0, then clearly $\sum m_j$ is odd, so condition 2 cannot hold.  Thus we need not concern ourselves with the case of eigenvalues with multiplicity where the eigenspace has vectors of both types.  

We remark also that condition 2 only puts a restriction on the eigenvalues corresponding to eigenvectors that do not vanish on $u$ or $v$.  If there is an eigenvector that vanishes on $u$ and $v$, then it need not be considered when trying to show that pretty good state transfer occurs in a graph.

Lemma \ref{lem:eig}, which gives spectral conditions for pretty good state transfer, is comparable to Lemma 2.1 in \cite{us}, which gives spectral conditions for perfect state transfer in a graph.  

Condition 1 in Lemma \ref{lem:eig} has a special name.  Two vertices $u,v$ for which $x(u) = \pm x(v)$ for every eigenvector $x$ of $H$ are called \emph{strongly cospectral} with respect to $H$.  Strong cospectrality is a strengthening of the notion of \emph{cospectral} vertices: vertices $u$ and $v$ are said to be cospectral if $G\setminus u$ and $G\setminus v$ are cospectral graphs.  More details on cospectral and strongly cospectral vertices can be found in \cite{godsil}.

\section{Graphs with an involution}

In the previous section, we saw in Lemma \ref{lem:eig} that pretty good state transfer in a graph can be characterized by the strong cospectrality of the vertices, and a number theoretic condition on the eigenvalues of the Hamiltonian.  We will see in this section that an involution in a graph automatically gives us the strong cospectrality condition, as long as the potential on the vertices respects the symmetry of the graph.  The involution also leads to a factorization of the characteristic polynomial of the Hamiltonian, so investigating pretty good state transfer on graphs with an involution comes down to studying the number theoretic properties of these factors to see when condition 2 of Lemma \ref{lem:eig} is satisfied.

Let $G$ be a graph with a non-trivial involution $\sigma$, and let $u,u'$ denote vertices of $G$ that are ``symmetric" across the involution, that is, $u' = \sigma(u)$.    Let $Q:V(G)\rightarrow \R$ be a potential on the vertex set of $G$, and throughout this section, we will require $Q$ to respect the involution $\sigma$ on $G$, namely we require $Q(u) = Q(u')$ for all $u\in V(G)$. Let $S=\{x\in V(G):\sigma x = x\}$ be the set of vertices fixed by $\sigma$.  We say that an edge, $uv\in E(G)$ is fixed by $\sigma$ if $v=\sigma u$. Throughout the rest of the paper, let $N=|V(G)|$ denote the number of vertices of $G$, and let us write $N = 2n+|S|$, so that $|S|$ is the number of fixed vertices, and $n$ is the number of vertices on either ``side" of the involution.  Finally, let us denote by $G'$ the ``half-graph" induced by the involution. More precisely, $G'$ is obtained from $G\setminus S$ choosing exactly one of $(x,\sigma x)$ for each $x$, and taking the induced subgraph of $G$ on the chosen vertices.  Note that there are (at least) two copies of $G'$ contained in $G$ on either side of the involution. We naturally put a potential $Q'$ on the vertices of $G'$ by restricting $Q$ to $G'$.

Given any $G$ and $G'$ as described above, then with the appropriate labeling of the vertices, the Hamiltonian $H$ can be written as follows:
\[
H = \begin{bmatrix}
H'&A_\sigma&A_S\\A_\sigma&H'&A_S\\A_S^T&A_S^T&H_S
\end{bmatrix}.
\]
Here, $H'$ denotes the $n\times n$ Hamiltonian of $G'$, $H_S$ denotes the $|S|\times|S|$ restriction of $H$ to the fixed set $S$, $A_\sigma$ denotes the part of the adjacency matrix for the edges with endpoints on both sides of the involution, and $A_S$ denotes the part of the adjacency matrix with the edges going from $S$ to its complement.  Note that $A_\sigma$ does not need a transpose in the second row, because $A_\sigma$ is a symmetric matrix because of the involution.  Indeed, an edge fixed by $\sigma$ will be on the diagonal of $A_\sigma$, and edges not fixed by $\sigma$ will come in pairs that are interchanged by $\sigma$.  These will occur in symmetric off-diagonal positions in $A_\sigma$.

To illustrate the above terminology and notation, let us give a small example. Let $G$ be the graph from Figure \ref{fig:invol}.  Note that $G$ has an involution $\sigma$ given by the symmetry in the drawing of the figure, namely $\sigma$ interchanges $v_1$ with $v_4$, $v_2$ with $v_5$, and $v_3$ with $v_6$, and fixes $v_7$.  The set $S$ of fixed vertices consists of the single vertex $v_7$.  The edge $v_3v_6$ is a fixed edge of the involution.  The edges $v_1v_5$ and $v_2v_4$ are not fixed edges, but are interchanged by $\sigma$.  Here $N=7$, and $n=3$, and the graph $G'$ is simply a path on 3 vertices (this can be thought of as the graph induced by $v_1,v_2,v_3$, or the graph induced by $v_4,v_5,v_6$ in $G$).

\tikzstyle{every node} = [circle, fill=black, inner sep=0pt, minimum width=4pt]
\begin{figure}[h]
\begin{center}
\begin{tikzpicture}
\draw (-1,0)--(1,0)(-1,2)--(1,1) (1,2)--(-1,1)(-1,2)node{}--(-1,1)node{}--(-1,0)node{}--(0,-1)node{}--(1,0)node{}--(1,1)node{}--(1,2)node{};
\draw (-1,2)node[left,fill=none]{$v_1$}(-1,1)node[left,fill=none]{$v_2$} (-1,0)node[left,fill=none]{$v_3$} (0,-1)node[below,fill=none]{$v_7$} (1,0)node[right,fill=none]{$v_6$}  (1,1)node[right,fill=none]{$v_5$}
(1,2)node[right,fill=none]{$v_4$};
\end{tikzpicture}
\end{center}
\caption{A graph with an involution.}\label{fig:invol}
\end{figure}
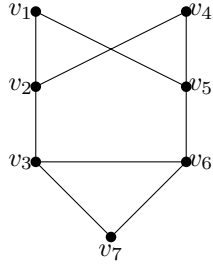

The Hamiltonian for the graph in Figure \ref{fig:invol} is given by
\[
H = \left[\begin{array}{ccc|ccc|c}
Q_1&1&0&0&1&0&0\\1&Q_2&1&1&0&0&0\\0&1&Q_3&0&0&1&1\\ \hline 0&1&0&Q_1&1&0&0\\1&0&0&1&Q_2&1&0\\0&0&1&0&1&Q_3&1\\ \hline 0&0&1&0&0&1&Q_4
\end{array}\right]
\]
with the partitions giving the block form described above.

With this terminology and notation, we can give the following.

\begin{lemma}\label{lem:invol_factor}
Let $G$ be a graph with an involution $\sigma$.  Then the characteristic polynomial $P(x)$ of the Hamiltonian $H$ of $G$ factors into two factors $P_+(x)$ and $P_-(x)$ which are, respectively, the characteristic polynomials of $H_+ := \begin{bmatrix}H'+A_\sigma &A_S\\2A_S^T&H_S\end{bmatrix}$ and $H_- := H' - A_\sigma$, where each of these matrices is as defined above.  

Furthermore, the eigenvectors of $H$ take the form $[a ~ a ~ b]^T$ and $[c ~ -c ~ 0]^T$ where $[a ~ b]^T$ is an eigenvector for $H_+$, and $c$ an eigenvector for $H_-$.
\end{lemma}
\begin{proof}
Suppose that we have
\[
\begin{bmatrix}
H'+A_\sigma&A_S\\2A_S^T&H_S
\end{bmatrix}\begin{bmatrix}
a\\b
\end{bmatrix} = \lambda\begin{bmatrix}
a\\b
\end{bmatrix}.
\]
Then $(H'+A_\sigma)a + A_Sb = \lambda a$ and $2A_S^Ta+H_Sb = \lambda b$.  Therefore
\[
\begin{bmatrix}
H'&A_\sigma&A_S\\A_\sigma&H'&A_S\\A_S^T&A_S^T&H_S
\end{bmatrix}\begin{bmatrix}
a\\a\\b
\end{bmatrix}=\begin{bmatrix}
(H'+A_\sigma)a+A_Sb\\(H'+A_\sigma)a+A_Sb\\2A_S^Ta+H_Sb
\end{bmatrix}=\lambda\begin{bmatrix}
a\\a\\b
\end{bmatrix}.
\]
Therefore any eigenvalue of $H_+$ (including multiplicity) is an eigenvalue of $H$, so its characteristic polynomial, $P_+(x)$ divides the characteristic polynomial $P(x)$.

Similarly, if $(H'-A_\sigma)c = \mu c$, then
\[
\begin{bmatrix}
H'&A_\sigma&A_S\\A_\sigma&H'&A_S\\A_S^T&A_S^T&H_S
\end{bmatrix}\begin{bmatrix}
c\\-c\\0
\end{bmatrix}=\begin{bmatrix}
(H'-A_\sigma)c\\-(H'-A_\sigma)c\\0
\end{bmatrix}=\mu\begin{bmatrix}
c\\-c\\0
\end{bmatrix}.
\]
Therefore any eigenvalue of $H_-$ is an eigenvalue of $H$ including multiplicity, so $P_-(x)$ divides $P(x)$.  

It is not hard to see that all the eigenvectors we have found are linearly independent, and clearly $P_+$ has degree $n+|S|$ and $P_-$ has degree $n$, so their product has degree $2n+|S| = N$, so we have found all the eigenvalues of $H$.
\end{proof}

We remark that, due to the form of the eigenvectors that we have found, we see that any graph with an involution has eigenvectors satisfying the strong cospectrality condition (condition 1 of Lemma \ref{lem:eig}).  Furthermore, the two different groups of eigenvalues from Lemma \ref{lem:eig}, the $\lambda_i$'s and the $\mu_j$'s, are the roots of $P_+$ and $P_-$ respectively.

Let us return to the example from Figure \ref{fig:invol}.  For this graph, we have 
\[
H_+ = \begin{bmatrix}
Q_1&2&0&0\\2&Q_2&1&0\\0&1&Q_3+1&1\\0&0&2&Q_4
\end{bmatrix} \text{ and } H_- = \begin{bmatrix}
Q_1&0&0\\0&Q_2&1\\0&1&Q_3-1
\end{bmatrix}.
\]
The characteristic polynomial (taking all the $Q_i$ to be 0 for simplicity of writing), is
\[
P(x) = x^7 - 9x^5-2x^4+19x^3+4x^2 - 8x = (x^4-x^3-7x^2+4x+8)(x^3+x^2-x)=P_+(x)P_-(x).
\]

We remark that Lemma 3.2 of \cite{us} shows explicitly the same factorization of the characteristic polynomial, and the corresponding $H_+$ and $H_-$ for the case of paths.

Note that in general, there is a slight ambiguity in defining $G'$ (and thus $H'$ and $A_\sigma$) in that we made an arbitrary choice for each pair $(x, \sigma x)$ of which of these to consider on which ``side" of the involution.  If we made a different choice, then it could change which non-fixed edges are in $G'$ and which go across the involution, changing whether an off-diagonal 1 shows up in $H'$ or $A_\sigma$. While this does not affect $H_+$, it may be the the sign of some of the off-diagonal entries of $H_-$ depend on this choice.  However, the spectrum remains unchanged, because, if we make a different choice for any particular pair, this changes the sign of the corresponding row and column in $H_-$.  This sign change can be achieved by a similarity of the matrix, so the spectrum is unaffected.  

\begin{lemma}\label{lem:invol_irred}
Let $G$ be a graph with a non-trivial involution $\sigma$ with fixed set $S$, and let us denote the number of vertices by $N=2n+|S|$ as before. Let us choose a potential $Q:V(G)\rightarrow\R$ that is symmetric with respect to $\sigma$.  Let us denote by $Q_1,...,Q_n$ the values of $Q$ that we are putting (symmetrically) on the non-fixed vertices of $G$, and $Q_{n+1},...,Q_{n+|S|}$ the values of $Q$ on the vertices in $S$.    Assume the following:
\begin{itemize}
\item The potential $Q$ is chosen so that $\bar{Q}=\sum_{i=1}^n Q_i$ is an irrational number, and $\hat{Q}=\sum_{i=1}^{|S|}Q_{n+i}$ is a number that is algebraically independent from $\bar Q$.  The number $\hat Q$ could be rational or irrational.
\item The polynomials $P_+$ and $P_-$ of Lemma \ref{lem:invol_factor} are irreducible over the base field $\Q(Q_1,...,Q_{n+|S|})$.
\end{itemize} 
If these two conditions are satisfied, then we have the following.
\begin{enumerate}
\item If $S\neq\emptyset$, then there is pretty good state transfer between $u$ and $\sigma u$ whenever $u\not\in S$.
\item If $S=\emptyset$ but there is at least one edge fixed by $\sigma$, then there is pretty good state transfer  between $u$ and $\sigma u$ for all $u$.
\end{enumerate}
\end{lemma}
\begin{proof}
Suppose we have integers $\ell_i$, $m_j$ satisfying \begin{align*}
\sum_i \ell_i\lambda_i +\sum_j m_j\mu_j = 0\\
\sum_i\ell_i +\sum_j m_j =0.
\end{align*}
To use Lemma \ref{lem:eig}, we wish to show 
$\sum_i \ell_i $
is even.

We will use a tool from Galois theory called the \emph{field trace} of a field extension.  For a Galois field extension $K$ of $F$, we define $\Tr_{K/F}: K \rightarrow F$ by 
\[
\Tr_{K/F}(\alpha) = \sum_{g\in Gal(K/F)}g(\alpha).
\]  The field trace is the trace of the linear map taking $x\mapsto \alpha x$.  We record here a few basic facts about the field trace that will be useful.
\begin{itemize}
\item $\Tr_{K/F}$ is a linear map.
\item For $\alpha\in F$, $\Tr_{K/F}(\alpha) = [K:F]\alpha$.
\item For $K$ an extension of $L$, and extension of $F$, we have $\Tr_{K/F} = \Tr_{L/F}\circ\Tr_{K/L}$.
\end{itemize}

Now, set $F=\Q(Q_1,...,Q_{n+|S|})$, let $L/F$ be the splitting field for $P_+$, $M/F$ the splitting field for $P_-$, and $K/F$ the smallest field extension containing both $L$ and $M$.  Since $P_+$ and $P_-$ are irreducible, then $L$ and $M$ are Galois extensions of $F$. We will examine the field trace of the individual roots of $P_+$ and $P_-$.  Let $\bar Q = \sum_{i=1}^n Q_i$ and let $k$ be the number of edges fixed by $\sigma$.  Observe that we have the traces of the matrices, $\Tr(H_+) = \bar Q +\hat Q+ k$ and $\Tr(H_-) = \bar Q - k$.  We have 
\[
\Tr_{L/F}(\lambda_i) = \sum_{g\in Gal(K/F)}g(\lambda_i)
\]
and since $L$ is a Galois extension, the group acts transitively on the roots of $P_+$, so each of the $\lambda_j$'s shows in this sum, and each will appear $|Gal(L/F)|/deg(P_+)$ times.  Thus
\[
\Tr_{L/F}(\lambda_i) = \frac{[L:F]}{n+|S|}\sum_j\lambda_j.
\]
Note further that $\sum_j\lambda_j$ is given by the trace of $H_+$ which we observed was $\bar Q + \hat Q+k$, so we have shown
\[
\Tr_{L/F}(\lambda_i) = \frac{[L:F]}{n+|S|}(\bar Q+\hat Q+k)
\]
for any $i$, In a similar way, by examining $P_-$ we obtain
\[
\Tr_{M/F}(\mu_j) = \frac{[M:F]}{n}(\bar Q-k)
\]
for any $j$.  

Now apply the field trace to our linear combination of the $\lambda_i$ and $\mu_j$, and using the properties above we have,
\begin{equation}\label{eq:fieldtrace}
\begin{array}{ll}
0&=\Tr_{K/F}\left(\sum\ell_i\lambda_i+\sum m_j\mu_j\right)\\ &= \Tr_{K/F}\left(\sum\ell_i\lambda_i\right)+\Tr_{K/F}\left(\sum m_j\mu_j\right)\\
&= [K:L]\Tr_{L/F}\left(\sum\ell_i\lambda_i\right)+[K:M]\Tr_{M/F}\left(\sum m_j\mu_j\right)\\
&=[K:L]\sum \ell_i\Tr_{L/F}(\lambda_i) + [K:M]\sum m_j\Tr_{M/F}(\mu_j)\\
&=\frac{[K:L][L:F]}{n+|S|}(\bar Q+\hat Q+k)\sum\ell_i + \frac{[K:M][M:F]}{n}(\bar Q-k)\sum m_j\\
&=[K:F]\left(\frac{\bar Q+\hat Q+k}{n+|S|}\sum\ell_i + \frac{\bar Q-k}{n}\sum m_j\right).
\end{array}
\end{equation}

Using this setup, we can apply these ideas to each case in the statement of the lemma.

\begin{enumerate}

\item Assume $S\neq\emptyset$, so $|S|>0$. Then (\ref{eq:fieldtrace}) implies that
\[
n(\bar Q+\hat Q+k)\sum\ell_i + (n+|S|)(\bar Q-k)\sum m_j = 0.
\]
Since $\bar Q$ is irrational and algebraically independent from $\hat Q$ and everything else is rational, then the coefficient in front of $\bar Q$ must be 0.  This implies that
\[
n\sum\ell_i + (n+|S|)\sum m_j =0
\]
which in turn implies that
\[
n\left(\sum\ell_i + \sum m_j\right) + |S|\sum m_j = 0.
\]
By assumption, we have that $\sum\ell_i + \sum m_j=0$, so since $|S|\neq0$, this gives $\sum m_j=0$, which further gives $\sum\ell_i=0$.  In particular, each sum is even, so Lemma \ref{lem:eig} implies that we get pretty good state transfer.  This completes the case that $S\neq\emptyset$.

\item Assume $S=\emptyset$ (implying, in particular, that $|S|=0$ and $\hat Q=0$) and $k\neq0$.  Then (\ref{eq:fieldtrace}) implies that 
\[
\left(\bar Q\left(\sum\ell_i+\sum m_j\right) +k\left(\sum\ell_i-\sum m_j \right)\right)
\]
By assumption, $\sum\ell_i+\sum m_j=0$, so this gives that $\sum\ell_i - \sum m_j = 0$, which, together with $\sum\ell_i +\sum m_j = 0$, gives that
\[
\sum\ell_i = \sum m_j =0.
\]
In particular, the sum is even, as desired.  This completes the proof.
\Hidden{
\item Assume that $S=\emptyset$, $k=0$, $n$ is even, and that $L\cap M= F$.  From above, we have
\[
\Tr_{L/F}\left(\sum\ell_i\lambda_i\right) = [L:F]\frac{\bar Q}{n}\sum\ell_i.
\]
On the other hand, under the assumption that $\sum\ell_i\lambda_i + \sum m_i \mu_i = 0$, we see that $\sum \ell_i\lambda_i  = -\sum m_i\mu_i\in M$.  So $\sum \ell_i\lambda_i \in L\cap M = F$.  Therefore, we know that
\[
\Tr_{L/F}\left(\sum\ell_i\lambda_i\right) = [L:F]\sum\ell_i\lambda_i.
\]
Thus, we can conclude that 
\[
\sum \ell_i\lambda_i = \frac{\sum\ell_i}{n}\bar Q.
\]
But clearly $\sum \ell_i\lambda_i$ is an algebraic integer over the field $F = \Q(Q_1,...,Q_n)$, therefore $\bar Q\sum\ell_i/n$ is as well.  This implies that $\sum\ell_i$ is divisible by $n$, which is an even number.  So $\sum\ell_i$ is even, which implies $\sum m_i$ is even, so we get pretty good state transfer by Lemma \ref{lem:eig}.  This completes the lemma.
}
\end{enumerate}
\end{proof}

With a little bit of work, if we add a more strict condition on the potential, then we can in fact reduce the requirement that both $P_+$ and $P_-$ are irreducible, to requiring that just $P_+$ be irreducible.  We present this in the following lemma.

\begin{lemma}\label{lem:one_irred}
Let $G$ be a graph with an involution $\sigma$, and assume all of the notation that we have been using previously.  Let $Q_1,...,Q_n,Q_{n+1},...,Q_{n+|S|}$ be values of the potential that are all algebraically independent transcendental numbers.  Assume the $P_+$ is an irreducible polynomial.  Then we have:
\begin{enumerate}
\item If $S\neq\emptyset$, then there is pretty good state transfer between $u$ and $\sigma u$ for $u\not\in S$.
\item If $S=\emptyset$ and there is at least one edge fixed by $\sigma$, then there is pretty good state transfer between $u$ and $\sigma u$ for all $u$.
\end{enumerate}
\end{lemma}
\begin{proof}
\begin{enumerate}
\item[]
\item When $S\neq\emptyset$, then from reasoning similar to (\ref{eq:fieldtrace}) in Lemma \ref{lem:invol_irred}, since we are assuming $P_+$ is irreducible, we can see that
\[
\frac{[K:F]}{n+|S|}(\bar Q + \hat Q + k)\sum\ell_i + \Tr_{K/F}\left(\sum m_j\mu_j\right) = 0.
\]
Since we have chosen all the $Q_i$'s to be algebraically independent transcendental numbers, then it is clear that $\hat Q$ is algebraically independent from everything else in this expression.  Thus, for this to be 0, we need the coefficient in front of $\hat Q$ to be 0.  In particular, we see that $\sum\ell_i = 0$, and we are done by Lemma \ref{lem:eig}.
\item Let us assume that $P_+$ is irreducible, and assume that $P_-$ factors into irreducible factors
\[
P_-(x) = P_{I_1}(x)P_{I_2}(x)...P_{I_r}(x)
\]
for some irreducible polynomials $P_{I_1},...,P_{I_r}$.  Here $I_1,...,I_r$ can be taken to be index sets that partition ${1,...,n}$, since $deg P_- = n$.  Since the coefficient of $x^{n-1}$ in $P_-$ is $\sum_{i=1}^n Q_i - k$, then we can see that in $P_{I_t}$, the coefficient of $x^{deg P_{I_t} - 1}$ is $\sum_{i\in I_t} Q_i + \alpha_t$ for some rational $\alpha_t$. Indeed, if some $Q_j$ were to show up in two distinct factors, then on multiplying them out, there would be some coefficient in $P_-$ that would involve a higher power of $Q_j$, which is impossible.  Note also that $\sum_{t=1}^r \alpha_t = -k$.

Let us now apply the field trace argument from Lemma \ref{lem:invol_irred} in this case.  Equation (\ref{eq:fieldtrace}) now becomes
\begin{align*}
0&=\frac{[K:F]}{n}\left(\sum_{i=1}^n Q_i + k\right)\sum\ell_i + [K:M]\sum_{t=1}^k\left(\sum_{j\in I_t}m_j\frac{[M:F]}{deq P_{I_t}}\left(\sum_{s\in I_t}Q_s + \alpha_t\right) \right)\\
&=\frac{[K:F]}{n}\left(\sum_{i=1}^n Q_i + k\right)\sum\ell_i+[K:F]\sum_{t=1}^k\left(\sum_{s\in I_t}Q_s + \alpha_t\right)\left(\frac{\sum_{j\in I_t}m_j}{|I_t|}\right).
\end{align*}
Now, the $Q_i$'s were chosen to be algebraically independent transcendental numbers, and as explained above, each appears exactly once in each of the two terms of the sum above.  Thus, the coefficient of $Q_i$ must be 0, so we obtain that 
\[
\frac{\sum_{i=1}^n\ell_i}{n} = -\frac{\sum_{j\in I_t}m_j}{|I_t|}
\]
for every $t$.  Replacing this in the above, we see that
\[
\frac{\sum\ell_i}{n}\left(\sum_{i=1}^n Q_i + k - \sum_{t=1}^k\left(\sum_{s\in I_t}Q_s + \alpha_t\right)\right) = 0.
\]
We saw that each $Q_i$ appears once in the factorization, so all the $Q_i$'s cancel, and $\sum\alpha_t = -k$, so we have $2k\sum\ell_i/n = 0$.  In particular, since $k$, the number of fixed edges of the involution, is non-zero, we get $\sum\ell_i = 0$, and we are done by Lemma \ref{lem:eig}.
\end{enumerate}
\end{proof}

A natural question to ask now is when the conditions of Lemmas \ref{lem:invol_irred} and \ref{lem:one_irred} can be satisfied.  This is the subject of what follows.

\begin{lemma}\label{lem:irred}
Let $G$ be a connected graph with an involution $\sigma$, and let $Q_1,...,Q_{n+|S|}$ be values of potential that we put on the vertices of $G$ symmetrically across $\sigma$ as before.  If $Q_1,...,Q_{n+|S|}$ are chosen to be algebraically independent transcendental numbers, then $P_+$ is irreducible over the field $\Q(Q_1,....,Q_{n+|S|})$.  
\end{lemma}
\begin{proof}
First, note that if $G$ is connected, then $H_+$ corresponds to a connected graph.

Let us suppose that $P_+$ factors into two non-trivial factors as
\[
P_+ = R\cdot \tilde R.
\]
over the field $\mathcal{F} = \Q(Q_1,\dots,Q_{n+|S|})$. The field $\mathcal{F}$ is isomorphic to the quotient field of the formal polynomial ring $\mathcal{R} = \Q[Q_1,\dots,Q_{n+|S|}]$. Since this is a unique factorization domain, by the Gauss Lemma we can assume that $R,\tilde R \in \mathcal{R}$. Let us write 
\[
R = \sum_{T\subset I}\left(\prod_{i\in T}Q_i\right)r_T ~~~\text{ and } ~~~\tilde R = \sum_{\tilde T\subset \tilde I}\left(\prod_{i\in \tilde T}Q_i\right)r_{\tilde T}
\]
where $I$ and $\tilde I$ are disjoint subsets of the index set with $I\cup \tilde I = [n+|S|]$, and $r_T$ and $r_{\tilde T}$ are polynomials depending on $T$ and $\tilde T$.  

First we will argue that if there is such a factorization, then neither $I$ nor $\tilde I$ is empty.  Without loss of generality, let us suppose that $\tilde I = \emptyset$ (this means that the only terms involving any of the $Q_i$'s are in $R$ so that $\tilde R$ is a rational polynomial).  This means that the only term showing up in $\tilde R$ is $\tilde r_\emptyset$.  If we write 
\[
P_+ = \sum_{U\subset [n]}\left(\prod_{i\in U} Q_i\right)p_U
\]
then we see that $\tilde r_{\emptyset}$ divides $p_U$ for all $U$.  In particular, $\tilde r_{\emptyset}|p_{[n]}$, but since each $Q_i$ is transcendental (in particular, each $Q_i$ is nonzero), it is clear that $p_{[n]} = 1$.  This is a contradiction, so we see that if $P_+$ factors, then both $I$ and $\tilde I$ are nonempty.  

Now, let us take $i\in I$, $j\in \tilde I$. Looking at the factorization $P_+ = R\cdot \tilde R$, since the $Q_i$'s are algebraically independent transcendentals, it is clear that 
\begin{align*}
p_\emptyset &= r_{\emptyset}\tilde r_{\emptyset}\\
p_{\{i\}} &= r_{\{i\}}\tilde r_{\emptyset}\\
p_{\{j\}} &= r_{\emptyset}\tilde r_{\{j\}}
\end{align*}
Now, $p_\emptyset$ is the characteristic polynomial of the matrix $H_+$ if all the $Q_i$ were taken to be 0 and $p_{\{i\}}$ and $p_{\{j\}}$ are the characteristic polynomials of the principal submatrices of this matrix obtained from deleting row and column $i$ and $j$ respectively.  We see from the above system that any root of $p_\emptyset$ is either a root of $p_{\{i\}}$ or $p_{\{j\}}$.  Thus any eigenvalue of this matrix is also an eigenvalue of one of its principal submatrices.  We know however that the eigenvalues interlace, so we see that the interlacing is not strict.  This implies that the associated eigenvector is 0 on the entry corresponding to the deleted row and column.  Thus, we have seen that every eigenvector has some entry equal to 0.  However, since the matrix corresponding to $P_+$ is non-negative, and its corresponding graph is connected, then the Perron-Frobenius theorem tells us that there is an eigenvector whose entries are all positive.  This is a contradiction, and thus $P_+$ is irreducible.
\Hidden{
For $P_-$, there are possibly some negative entries (corresponding to edges that have endpoints on either side of the involution), so the Perron-Frobenius argument will not work.  However, we can use a more general argument. Note that for any $S\subset [n]$, we have that 
\[
p_T = r_{I\cap T}\tilde r_{\tilde I \cap T}
\]
so that in particular, 
\begin{align*}
p_I &= r_I\tilde r_\emptyset\\
p_{\tilde I} &= \tilde r_{\tilde I}r_\emptyset.
\end{align*}
This implies that
\[
p_Ip_{\tilde I} = r_I\tilde r_{\tilde I}r_\emptyset\tilde r_\emptyset = p_{[n]}p_\emptyset = p_\emptyset
\]
since, again, clearly $p_{[n]}=1$.  But as before, $p_\emptyset$ is the characteristic polynomial of the matrix where the potential is taken to be 0.  Also, note that $p_I$ and $p_{\tilde I}$ are characteristic polynomials of disjoint proper principal submatrices of the same matrix.  Thus we have shown that this characteristic polynomial factors as the product of the characteristic polynomials of two proper principal submatrices.  We claim that this implies that the matrix itself is the direct sum of these two principal submatrices, which in turn implies that the graph in question is disconnected. 
}
\end{proof}

From Lemmas \ref{lem:one_irred} and \ref{lem:irred}, the proof of Theorem \ref{thm:PGSTinvol} follows immediately.  We have shown that for any graph with an involution that fixes something, we can choose a potential so that we get pretty good state transfer across the involution.


A natural question to ask at this point is how simple of a potential we can use to obtain pretty good state transfer, or if we need the full generic potential on all the vertices.  In particular, if we have a graph with an involution, it is of interest to determine if there is a single value that we can place on a symmetric pair of vertices (and a potential of 0 elsewhere) to induce pretty good state transfer between those vertices.

Let us thus consider this setup.  Let $Q:V(G)\rightarrow \R$ satisfy $Q(x) = 0$ unless $x=u$ or $x=u'$ for some fixed symmetric pair $u,u'$.  By abuse of notation, we will simply say $Q(u) =Q(u') = Q$.  Note then that we can write
\[
P_+(x) = p_1(x) - Qq_1(x) \text{ and } P_-(x) = p_2(x) - Qq_2(x)
\]
for some polynomials $p_1,q_1,p_2,q_2$.  Further more $p_1,p_2$ are, respectively, the characteristic polynomials of $H_+$ and $H_-$ where we take $Q=0$, and $q_1,q_2$ are the characteristic polynomials of these matrices with the row and column corresponding to $u$ deleted.

Then we have the following special case of Lemma \ref{lem:invol_irred}.

\begin{corollary}\label{cor:invol_single}
Let $G$ be a graph with an involution and $u,u'$ non-fixed vertices that are symmetric, and put a potential $Q$ only on $u$ and $u'$.  Let $p_1,p_2,q_1,q_2$ be as above.
Suppose that $(p_1,q_1)=1$ and that $(p_2,q_2)=1$ and choose $Q$ to be transcendental. Let $S$ be the set of vertices fixed by $\sigma$, and let $k$ be the number of edges fixed by $\sigma$.  Then if $\sigma$ fixes at least one vertex or at least one edge, then there is pretty good state transfer from $u$ to $u'$.
\end{corollary}
\begin{proof}
By Lemma \ref{lem:invol_irred}, since we are choosing $Q$ to be transcendental and it is the only value of the potential that we are using, the only thing we need to check is if $P_+$ and $P_-$ are irreducible.  However, as we saw above, $P_+ = p_1+Qq_1$ and $P_- = p_2 + Qq_2$.  

Note that any polynomial of the form $P(x) +Q\cdot R(X)$ will be irreducible over $\Q(Q)$ if $(P,R)=1$.  This follows since the polynomial is linear in $Q$, so if it factors, at least one factor must belong to $\Q[x]$, in which case it would have to be a factor of both $P$ and $R$. 

By assumption, $p_1$ and $q_1$ share no common factors, and so since $Q$ is transcendental, there can be no factorization of $P_+$.  Similarly for $P_-$.  This gives the result.
\end{proof}

Note that in the above, we have chosen $Q$ to be transcendental so that $P_+$ and $P_-$ will be irreducible.  It is interesting to ask if this is necessary to guarantee irreducibility. Hilbert's irreducibility theorem suggests that in fact, there are infinitely many rational values of $Q$ that would yield irreducible polynomials, but we do not know  how to describe all of them.

\section{Pretty good state transfer on paths}

We will apply the results of the previous section to paths to prove Theorem \ref{thm:PGSTpath}, which we restate here.

\begin{theorem}\label{thm:path}
Let $P_N$ denote the path on $N$ vertices and let $u$ and $v$ be the endpoints of the path.  Place a potential on $P_N$ with value $Q$ on $u$ and $v$, and 0 on all the other vertices.  Then for any $N$, there is a choice of $Q$ for which pretty good state transfer occurs between $u$ and $v$.
\end{theorem}

\begin{proof}
Clearly $P_N$ has an involution, so we obtain the factorization of the characteristic polynomial into $P_+P_-$ as in Lemma \ref{lem:invol_factor}.  We wish to obtain a more specific form for $P_+$ and $P_-$.

Observe that if $N=2n$ is even, then the involution of the path has a single fixed edge, and no fixed vertices, and if $N=2n+1$ is odd, then the involution of the path has a single fixed vertex and no fixed edges.

Let $p_N(x)$ denote that characteristic polynomial of the \emph{adjacency} matrix (no potential) of $P_N$.  Examining the matrices that we obtain from Lemma \ref{lem:invol_factor} for paths, it is not hard to see that, if $N=2n$ is even then 
\begin{equation}\label{eq:even}
\begin{array}{cl}
P_+(x) &= (p_n(x) - p_{n-1}(x)) - Q(p_{n-1}(x)-p_{n-2}(x))\\
P_-(x) &= (p_n(x) + p_{n-1}(x)) - Q(p_{n-1}(x)+p_{n-2}(x))
\end{array}
\end{equation}
and if $N=2n+1$ is odd, then
\begin{equation}\label{eq:odd}
\begin{array}{cl}
P_+(x) &= (p_{n+1}(x)-p_{n-1}(x)) - Q(p_n(x)-p_{n-2}(x))\\
P_-(x) &= p_n(x) - Qp_{n-1}(x).
\end{array}
\end{equation}

Let us choose $Q$ to be any transcendental real number.
Since an even path has a fixed edge, and an odd path has a fixed vertex under its involution, then by Corollary \ref{cor:invol_single}, all we need to show is that the relevant polynomials are relatively prime.  Namely, we must show that $(p_n-p_{n-1},p_{n-1}-p_{n-2})=1$, $(p_n+p_{n-1},p_{n-1}+p_{n-2})=1$, $(p_{n+1}-p_{n-1},p_{n}-p_{n-2})=1$, and $(p_n,p_{n-1})=1$, and .

Let us first focus on $(p_n\pm p_{n-1},p_{n-1}\pm p_{n-2})$.  Note from (\ref{eq:even}) and Lemma \ref{lem:invol_factor}, if we take $Q=0$, then we see that $p_n\pm p_{n-1}$ is a factor of $p_{2n}$ and $p_{n-1}\pm p_{n-2}$ is a factor of $p_{2n-2}$.  So we will be done if we can show that $p_{2n}$ and $p_{2n-2}$ do not share any roots.  Recall that $p_{2n}$ and $p_{2n-2}$ are characteristic polynomials for the adjacency matrix of a path.  The roots of these are well known.  The roots of $p_{2n}$ are \[2\cos\left(\frac{k\pi}{2n+1}\right),~k=1,...,2n\] and the roots of $p_{2n-2}$ are \[2\cos\left(\frac{k\pi}{2n-1}\right),~k=1,...,2n-2.\] Note that these are the $x$-coordinates of $2n$ and $2n-2$ evenly spaced out points along a circle centered at the origin respectively.  Since $2n+1$ and $2n-1$ are adjacent odd numbers, it is clear that these do not coincide.  Thus, these polynomials do not share any roots, so they are relatively prime.

The same reasoning that we used at the end of the last paragraph shows that $(p_n,p_{n-1})=1$.

For $(p_{n+1}+p_{n-1},p_{n}+p_{n-2})$, we can similarly see that these are factors of $p_{2n+1} = p_n(p_{n+1}+p_{n-1})$ and $p_{2n-1} = p_{n-1}(p_n + p_{n-2})$.  Now, in this case, $p_{2n+1}$ and $p_{2n-1}$ do in fact share one root, namely $x=0$, but no other roots (by reasoning similar to above).  However, if $x$ is a root of $p_n$, then it is not a root of $p_{n-1}$, and conversely, thus $x$ can be a root of exactly one of $p_{n+1}+p_{n-1}$ and $p_{n}+p_{n-2}$.  Thus these do not share any roots, and so are relatively prime as desired. 

Thus we have showed that all the necessary polynomials are relatively prime.  Thus we get pretty good state transfer by Corollary \ref{cor:invol_single}.
\end{proof}


\section{Further Directions}

It is natural to ask what happens for a graph with an involution that does not fix any vertices or edges.  This situation is much more subtle.  The field trace argument that we used in the proofs of Lemmas \ref{lem:invol_irred} and \ref{lem:one_irred} will not yield any information here, since both $P_+$ and $P_-$ have the same degree and will have identical traces.  We can reason through some special cases to obtain some partial results.  We find the following somewhat surprising since the parity of $n$ seems to play a large role in determining if pretty good state transfer can occur.

\begin{proposition}\label{prop:evenodd}
Let $G$ be a graph with an involution $\sigma$ that has no fixed vertices or edges, and let the number of vertices of $G$ be $N=2n$.  Let $Q$ be a potential on $G$ that is symmetric across $\sigma$.  Let $P_+$ and $P_-$ be the polynomials that we have had as previously.  Then we have the following.
\begin{enumerate}
\item If $n$ is even, and if $P_+$ is irreducible and the splitting fields for $P_+$ and $P_-$ intersect only in the base field, then there is pretty good state transfer between $u$ and $\sigma u$ for all $u$.
\item If $n$ is odd, and if all the eigenvectors of $H$ are non-vanishing on vertices $u,\sigma u$, the there is not pretty good state transfer between $u$ and $\sigma u$.
\end{enumerate}
\end{proposition}
\begin{proof}
\begin{enumerate}
\item  Let $Q_1,...,Q_n$ be the values of the potential, and let $\bar Q = \sum_i Q_i$. Let $F$ be the base field $\Q(Q_1,...,Q_n)$, $L$ the splitting field for $P_+$ and $M$ the splitting field for $P_-$ and assume that $L\cap M= F$.  From the field trace tools used previously, we have
\[
\Tr_{L/F}\left(\sum\ell_i\lambda_i\right) = [L:F]\frac{\bar Q}{n}\sum\ell_i.
\]
On the other hand, under the assumption that $\sum\ell_i\lambda_i + \sum m_i \mu_i = 0$, we see that $\sum \ell_i\lambda_i  = -\sum m_i\mu_i\in M$.  So $\sum \ell_i\lambda_i \in L\cap M = F$.  Therefore, we know that
\[
\Tr_{L/F}\left(\sum\ell_i\lambda_i\right) = [L:F]\sum\ell_i\lambda_i.
\]
Thus, we can conclude that 
\[
\sum \ell_i\lambda_i = \frac{\sum\ell_i}{n}\bar Q.
\]
But clearly $\sum \ell_i\lambda_i$ is an integral element over the ring $\mathcal{R} = \Z[Q_1,...,Q_n]$, therefore $\bar Q\sum\ell_i/n$ is as well.  This implies, via the Gauss Lemma, that $\sum\ell_i$ is divisible by $n$, which is an even number.  So $\sum\ell_i$ is even, which implies $\sum m_i$ is even, so we get pretty good state transfer by Lemma \ref{lem:eig}.

\item Since there are no fixed edges, then we have $\Tr(H_+) = \Tr(H_-) = \sum Q_i$, and since there are no fixed vertices, and no eigenvectors vanish on symmetric pairs of nodes, then the number of $\lambda_i$'s is the same as the number of $\mu_i$'s.  Thus, we can simply choose $\ell_i = 1$ for all $i$, and $\mu_i = -1$ for all $i$.  Then clearly we have $\sum\ell_i + \sum m_i = 0$, and
\[
\sum \ell_i\lambda_i + \sum m_i\mu_i = \sum Q_i - \sum Q_i = 0,
\]
but $\sum m_i$ and $\sum\ell_i $ are both odd, so we do not get pretty good state transfer, by Lemma \ref{lem:eig}.
\end{enumerate}
\end{proof}

We will give an example to illustrate some interesting aspects of Corollary \ref{cor:invol_single} and Proposition \ref{prop:evenodd}.

\begin{example}\label{ex:C6}
Let $G=C_6$, the cycle on six vertices.  Label the vertices $v_1,...,v_6$ cyclically around the cycle.  The question we wish to address is whether there is pretty good state transfer between $v_2$ and $v_5$.  Let us consider two involutions of $C_6$: first, let $\sigma$ be the map which interchanges $v_1$ with $v_6$, $v_2$ with $v_5$, and $v_3$ wtih $v_4$, and let $\sigma'$ be the map which interchanges $v_1$ with $v_4$, $v_2$ with $v_5$ and $v_3$ with $v_6$.  The involutions $\sigma$ and $\sigma'$ are, respectively, reflections across the vertical axis in the two different drawings of $C_6$ shown in Figure \ref{fig:c6}.  Note that $\sigma$ fixes two edges, $v_1v_6$ and $v_3v_6$, and $\sigma'$ does not fix any edges or vertices.  Notice also that both involutions have an odd number of vertices on either side.  Since both $\sigma$ and $\sigma'$ interchange $v_2$ with $v_5$, it would appear that the result of Proposition \ref{prop:evenodd} is at odds with the results of Theorem \ref{thm:PGSTinvol} and Corollary \ref{cor:invol_single}.  However, this is not the case.

First, notice that in Theorem \ref{thm:PGSTinvol}, the potential that is guaranteed by Lemma \ref{lem:irred} requires the potential to be non-zero everywhere, and symmetric across the involution.  Therefore the potential does encode the information of which involution we are looking at, that is, it ``breaks" the symmetry of the other involution.  Thus, our results say that if the potential is chosen to be symmetric across $\sigma$, then there is a choice of potential that gives pretty good state transfer from $v_2$ to $v_5$ (this potential will not be symmetric across $\sigma'$).  However, if the potential is chosen to be symmetric across $\sigma'$, then if the eigenvectors are all non-zero, then pretty good state transfer does not occur.  

Now let us consider this graph with a single value of $Q$ placed on $v_2$ and $v_5$ (this is symmetric with respect to both $\sigma$ and $\sigma'$).  In this case, both $P_+$ and $P_-$ factor, and there are eigenvectors that are 0 on $v_2$ and $v_5$, so none of our results are applicable.  We can look at the spectrum directly however. The characteristic polynomial in this case is.
\[
\left(x^2 - (Q+1)x+Q-2\right)\left(x^2-(Q-1)x-(Q+2)\right)(x-1)(x+1).
\]
The eigenvectors for $-1$ and $1$ are, respectively, $[-1,0,1,-1,0,1]^T$ and $[1,0,-1,-1,0,1]^T$, so $\pm1$ both are eigenvalues for which the entry corresponding to $v_2$ and $v_5$ are 0.  Thus, they do not need to be taken into account when considering the linear combinations from Lemma \ref{lem:eig}.  Thus, we have
\[
\lambda_1,\lambda_2 = \frac12\left(Q+1\pm\sqrt{Q^2-2Q+9}\right)~~~~\mu_1,\mu_2 = \frac12\left(Q-1\pm\sqrt{Q^2+2Q+9}\right)
\]
For the condition of Lemma \ref{lem:eig} to hold, we need to examine the system 
\begin{align*}
\ell_1\lambda_1 + \ell_2\lambda_2+m_1\mu_1+m_2\mu_2 &=0
\\
\ell_1+\ell_2+m_1+m_2 &=0.
\end{align*}
It is not hard to see that for generic $Q$, the only solution to this system is $\ell_1=\ell_2=m_1=m_2=0$, and thus we get pretty good state transfer by Lemma \ref{lem:eig}.  Note also that, if we had included $1$ and $-1$ in the linear combinations, we could have obtained solutions where $\sum\ell_i$ and $\sum m_j$ are odd.  Thus, it is important to only consider those eigenvectors that are non-vanishing on the vertices in question.  Furthermore, this example shows that there can be cases where the polynomials $P_+$ and $P_-$ factor, but we still get pretty good state transfer.

\begin{figure}
\begin{center}
\begin{tikzpicture}
\draw (-2,0)node{}--(-1,1)node{}--(1,1)node{}--(2,0)node{}--(1,-1)node{}--(-1,-1)node{}--(-2,0);
\draw (-1,1)node[left,fill=none]{$v_1$} (-2,0)node[left,fill=none]{$v_2$} (-1,-1)node[left,fill=none]{$v_3$} (1,-1)node[right,fill=none]{$v_4$} (2,0)node[right,fill=none]{$v_5$} (1,1)node[right,fill=none]{$v_6$};
\end{tikzpicture}~~~~
\begin{tikzpicture}
\draw (-2,0)node{}--(-1,1)node{}--(1,-1)node{}--(2,0)node{}--(1,1)node{}--(-1,-1)node{}--(-2,0);
\draw (-1,1)node[left,fill=none]{$v_1$} (-2,0)node[left,fill=none]{$v_2$} (-1,-1)node[left,fill=none]{$v_3$} (1,1)node[right,fill=none]{$v_4$} (2,0)node[right,fill=none]{$v_5$} (1,-1)node[right,fill=none]{$v_6$};
\end{tikzpicture}
\end{center}
\caption{Two involutions of $C_6$}\label{fig:c6}
\end{figure}
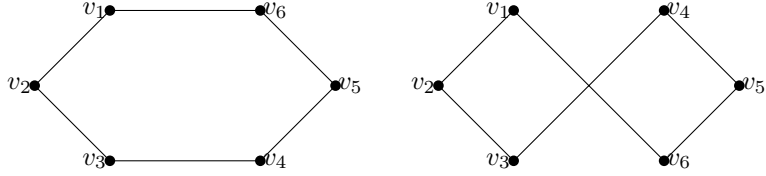

\end{example}

We end with a discussion of a few avenues of research that are still open.
\begin{enumerate}
\item \emph{Can we complete any cases missing from Proposition \ref{prop:evenodd}?}  The condition that the splitting fields for $P_+$ and $P_-$ have trivial overlap seems very unnatural, but is essential in the proof.  Furthermore, it is unclear how things work if we have eigenvectors that vanish on the vertices in question.  It would be of interest to determine if the potential can always be chosen to force the eigenvectors to be non-vanishing everywhere.  

\item \emph{Can we say anything when neither $P_+$ nor $P_-$ is irreducible?}  Corollary \ref{cor:invol_single} requires irreducibility as a condition, and Theorem \ref{thm:PGSTinvol} was proven by showing the potential can be chosen so that  $P_+$ is irreducible.  However, neither of these says that we do not get pretty good state transfer if $P_+$ does factor further.  See, for instance, Example \ref{ex:C6}.  It is of interest to investigate further what happens in general.

\item \emph{Can we get pretty good state transfer without an involution?}  We studied graphs with involutions because they automatically give rise to pairs of vertices that are strongly cospectral.  However, examples of strongly cospectral pairs are known that do not come from an involution.  Work on such examples does seem to indicate that we can still obtain a factorization of the characteristic polynomial, but the involution gave us a natural way to understand this factorization.  It is of interest to see what can be done in the absence of this symmetry.

\item \emph{When pretty good state transfer occurs, can we say anything about the time it takes to for the state transfer between two vertices to get within $\epsilon$ of perfect?}  Lemma \ref{lem:eig} gives conditions that guarantee the \emph{existence} of a $t$ for which we are within $\epsilon$, but we do not know how to find such a $t$.  Especially in the case where the potential takes a single non-zero value on a pair of nodes, it is of interest to determine if there is any interaction between the size of the potential and the time to achieve the maximum state transfer.
\end{enumerate}




\bibliographystyle{hplain}
\bibliography{quantum}

\end{document}